\documentclass[12pt]{amsart}
\usepackage[utf8]{inputenc}

\title[Weak compactness in Lipschitz-free spaces]{Weak compactness in Lipschitz-free spaces over superreflexive spaces}
\author[Z. Silber]{Zdeněk Silber}
\email{zdesil@seznam.cz}
\address{Institute of Mathematics of the Polish Academy of Sciences,
ul. \'Sniadeckich 8,  00-656 Warszawa, Poland}

\usepackage{amsthm}
\usepackage{amsmath}
\usepackage{mathrsfs}
\usepackage{mathtools}
\usepackage[shortlabels]{enumitem}
\usepackage{verbatim}
\usepackage{amsfonts}
\usepackage{amssymb}
\usepackage[a-2u]{pdfx}
\usepackage{tikz}
\usepackage[most]{tcolorbox}
\usepackage{esint}
\usepackage{enumitem}
\usetikzlibrary{matrix,arrows}
\graphicspath{{images/}}

\newcommand{\N}{\mathbb{N}}

\newcommand{\R}{\mathbb{R}}

\newcommand{\F}{\mathcal{F}}

\newtheorem{theorem}{Theorem}[section]

\newtheorem{prop}[theorem]{Proposition}
\newtheorem{corollary}[theorem]{Corollary}

\newtheorem{question}{Question}

\theoremstyle{definition}
\newtheorem{definition}[theorem]{Definition}
\newtheorem{remark}[theorem]{Remark}
\newtheorem{example}[theorem]{Example}

\newcommand{\norm}[1]{\left\lVert#1\right\rVert}

\begin{document}

\baselineskip=17pt

\begin{abstract}
    We show that the Lipschitz-free space $\mathcal{F}(X)$ over a superreflexive Banach space $X$ has the property that every weakly precompact subset of $\mathcal{F}(X)$ is relatively super weakly compact, showing that this space "behaves like $L_1$" in this context. As consequences we show that $\mathcal{F}(X)$ enjoys the weak Banach-Saks property and that every subspace of $\mathcal{F}(X)$ with nontrivial type is superreflexive. Further, weakly compact subsets of $\mathcal{F}(X)$ are super weakly compact and hence have many strong properties. To prove the result, we use a modification of the proof of weak sequential completeness of $\mathcal{F}(X)$ by Kochanek and Pernecká and an appropriate version of compact reduction in the spirit of Aliaga, No{\^u}s, Petitjean and Procházka.
\end{abstract}

\maketitle

\section{Introduction}

This paper is dedicated to strengthening some results concerning weak compactness in Lipschitz-free spaces over superreflexive spaces. Note that the Lipschitz-free space over the real line, $\mathcal{F}(\R)$, is isomorphic to $L_1$ (see Example 2.1. of \cite{Godefroy2015}). Structure of Lipschitz-free spaces over Banach spaces of higher dimension seems to be more complicated, even when the dimension is finite. However, some properties of $L_1$ are shared by Lipschitz-free spaces over Banach spaces of higher dimension. In \cite{CuthDouchaWojtasczyk2016}, the authors showed that for any $M \subseteq \R^n$, the Lipchitz-free space $\mathcal{F}(M)$ is weakly sequentially complete. Later, this result was extended in \cite{KochanekPernecka} for $M$ compact subset of a superreflexive space. Both of these results used a modification of a method by Bourgain \cite{Bourgain1983,Bourgain1984} (see Theorem III.D.31 of \cite{WojtaszczykBook}). Finally, in \cite{AliagaCompRed}, the authors proved a compact reduction theorem for weak sequential completeness, showing that $\mathcal{F}(X)$ is weakly sequentially complete if $X$ is a superreflexive Banach space.

The results above shed some light on the structure of Lipschitz-free spaces over superreflexive spaces. Weak sequential completeness gives us that a bounded subset $A$ of $\mathcal{F}(X)$ is relatively weakly compact if and only if it is weakly precompact (i.e. if and only if there is no sequence in $A$ equivalent to the canonical basis of $\ell_1$), and also implies that no non weakly sequentially complete Banach space (in particular $c_0$) can be isomorphic to a subspace of $\mathcal{F}(X)$.

A different area of study is the theory of super weakly compact sets. Super weakly compact sets are a localization of superreflexivity in the same sense as weakly compact sets are a localization of reflexivity. Many papers were devoted to this study \cite{Cheng2018,Cheng2010,GrelierRaja2022,LancienRaja2022,Raja2016,Kun2021}, and many interesting properties of super weakly compact sets are known (see Subsection 1.2. and references therein). There are spaces in which every weakly compact set is super weakly compact -- superreflexive spaces, Schur spaces, $L_1(\mu)$, $C(K)^*$ or preduals of JBW$^*$ triples (see Section 6 of \cite{LancienRaja2022}). The main result of this paper is that the Lipschitz-free space $\mathcal{F}(X)$ for $X$ superreflexive enjoys this property as well. Hence, in such $\mathcal{F}(X)$, it holds that a bounded set $A$ is weakly precompact if and only if it is relatively weakly compact if and only if it is relatively super weakly compact (we call this property of Banach spaces property (R), see Section 2). This gives us that any weakly compact subset of $\mathcal{F}(X)$ for $X$ superreflexive enjoys all of the strong properties of super weakly compact sets (see e.g. Propositions \ref{Proposition-SWCC-char} and \ref{Proposition-SWC-char}).

As mentioned in the above paragraph, we say that a Banach space $X$ has \textit{property (R)} if every weakly precompact subset of $X$ is relatively super weakly compact. In \cite{Rosenthal1973}, Rosenthal showed that every subspace of $L_1(\mu)$, for finite $\mu$, which does not contain $\ell_1$ embeds in $L_p(\mu)$ for some $p > 1$ (and such $L_p(\mu)$ is superreflexive as it is uniformly convex). Property (R) can be thus understood as a localization of the above result of Rosenthal, where we consider weakly precompact sets instead of subspaces not containing $\ell_1$ and relatively super weakly compact sets instead of superreflexive spaces. It is clear (see Propositon \ref{Proposition-R-char}) that a Banach space has property (R) if and only if it is weakly sequentially complete and every weakly compact subset of $X$ is super weakly compact. Property (R) is an intermediate property, stronger than weak sequential completeness but weaker than the Schur property (as being relatively compact implies being relatively super weakly compact, which in turn implies being relatively weakly compact). We devote Section 2 to the study of property (R).

Finally, once we show that $\mathcal{F}(X)$ has property (R) for superreflexive $X$, we can improve some non-embeddability results. In particular, we show that any subspace of such $\mathcal{F}(X)$ with nontrivial type is superreflexive (Proposition \ref{Proposition-Properties-R}). This implies that the Banach spaces $(\bigoplus_{N=1}^\infty \ell_1^N)_{\ell_p}$, $(\bigoplus_{N=1}^\infty \ell_\infty^N)_{\ell_p}$, for $p > 1$, cannot embed in $\mathcal{F}(X)$ (Corollary \ref{Corollary-R}). Finally, we show that any Banach space with property (R), and thus $\mathcal{F}(X)$ for superreflexive $X$, has the weak Banach-Saks property (Proposition \ref{Proposition-R-WeakBS}).

The paper is organised as follows. In the rest of the introductory section we recall all the key notions necessary for the paper, namely weak sequential completeness, Schur property, Lipschitz-free spaces and some propeties of relatively super weakly compact sets. In the second section we introduce the property (R), prove some of its properties and give examples. Third section is devoted to the proof that Lipschitz-free spaces over superreflexive spaces have property (R). In the last section we give a proof of compact reduction for the weak Banach-Saks property due to R. Aliaga, G. Grelier and A. Procházka and mention some open problems

\subsection{Weak sequential completeness and Schur property}

Let us begin with the definition.

\begin{definition}
    A sequence $(x_n)_{n \in \N}$ in a Banach space $X$ is \textit{weakly Cauchy} if for every $x^* \in X^*$ the sequence of scalars $(x^*(x_n))_{n \in \N}$ is convergent.
    A Banach space $X$ is \textit{weakly sequentially complete} if every weakly Cauchy sequence in $X$ is weakly convergent.
\end{definition}

Application of the uniform boundedness principle yields that a sequence $(x_n)_{n \in \N}$ in $X$ is weakly Cauchy if and only if there is some $x^{**} \in X^{**}$ such that $(x_n)_{n \in \N}$ converges to $x^{**}$ in the weak$^*$ topology of $X^{**}$. Clearly, every reflexive space is weakly sequentially complete. An example of a nonreflexive Banach space which is weakly sequentially complete is $L_1(\mu)$.

\begin{definition}
    A Banach space $X$ has the \textit{Schur property} (or $X$ is a \textit{Schur space}) if every weakly convergent sequence in $X$ is norm convergent.
\end{definition}

It is easy to see that in a Schur space every weakly Cauchy sequence is norm Cauchy, and thus Schur spaces are weakly sequentially complete. The typical example of a Banach space with the Schur property is $\ell_1$. Schur property is, in a sense, very far away from reflexivity -- if a Banach space is simultaneously reflexive and has the Schur property, then it must be finite dimensional (this follows from Eberlein-Šmulyan theorem).

Weak sequential completeness is a property which is common to both reflexive and Schur spaces. Recall the statement of Rosenthal's $\ell_1$ theorem \cite{Rosenthal1974}: \textit{Every bounded sequence in a Banach space $X$ has a subsequence which is either weakly Cauchy or equivalent to the canonical basis of $\ell_1$}. This theorem gives us a characterization of both weak sequential completness and Schur property. We first recall the definition of weakly precompact sets.

\begin{definition}
    A bounded subset $A$ of a Banach space $X$ is \textit{weakly precompact} if it does not contain a sequence equivalent to the canonical basis of $\ell_1$.
\end{definition}

By Rosenthal's $\ell_1$ theorem, a bounded set $A$ is weakly precompact if and only if every sequence in $A$ has a weakly Cauchy subsequence. Weakly precompact sets are preserved under convex hulls (see e.g. Addendum of \cite{Rosenthal1977}) and images under bounded linear operators (this follows from the fact that if a seminormalised sequence $(x_n)_{n \in \N}$ maps onto the canonical basis of $\ell_1$ by a bounded linear map, then $(x_n)_{n \in \N}$ must itself be equivalent to the canonical basis $\ell_1$). The following propositions follow immediately from the definitions, Rosenthal's $\ell_1$ theorem and Eberlein-Šmulyan theorem.

\begin{prop} \label{Proposition-WSC-char}
    Let $X$ be a Banach space. The following are equivalent:
    \begin{enumerate}
        \item $X$ is weakly sequentially complete;
        \item every bounded sequence in $X$ has a subsequence which is either weakly convergent or equivalent to the $\ell_1$ basis;
        \item every weakly precompact subset of $X$ is relatively weakly compact.
    \end{enumerate}
\end{prop}

\begin{prop} \label{Proposition-Schur-char}
    Let $X$ be a Banach space. The following are equivalent:
    \begin{enumerate}
        \item $X$ has the Schur property;
        \item every bounded sequence in $X$ has a subsequence which is either norm convergent or equivalent to the $\ell_1$ basis;
        \item every weakly precompact subset of $X$ is relatively compact.
    \end{enumerate}
\end{prop}

\subsection{Lipchitz-free spaces}

Let $(M,d)$ be a pointed metric space (that is, a metric space with a distinguished point $0 \in M$). We denote by $\operatorname{Lip}_0(M)$ the Banach space of Lipschitz functions on $M$ which map $0 \in M$ to $0 \in \R$, equipped with the Lipschitz-constant norm defined by formula
$$\norm{f}_{\operatorname{Lip}} = \sup \left\{ 
\frac{|f(p)-f(q)|}{d(p,q)}: p,q \in M, p \neq q \right\}.$$
Then $M$ embeds isometrically in the dual space $\operatorname{Lip}_0(M)^*$ via the evaluation map $\delta:M \rightarrow \operatorname{Lip}_0(M)^*$, $\langle \delta(p),f \rangle = f(p)$, $p \in M$, $f \in \operatorname{Lip}_0(M)$. The \textit{Lipschitz-free space} $\mathcal{F}(M)$ is then the closed linear span of $\delta(M)$ with the norm $\norm{\cdot}_\mathcal{F}$ induced by the dual norm of $\operatorname{Lip}_0(M)^*$.

The Lipschitz-free space $\mathcal{F}(M)$ is an isometric predual of $\operatorname{Lip}_0(M)$ and is characterized by the following universality property: For any Banach space $X$ and Lipschitz map $L:M \rightarrow X$ with $L(0) = 0$ there is a unique linear map $\tilde{L}: \mathcal{F}(M) \rightarrow X$ such that $\tilde{L} \circ \delta = L$ and $\|\tilde{L} \| = \norm{L}_{\operatorname{Lip}}$. For more information about Lipchitz-free spaces see the book \cite{Weaver1999} or the survey paper \cite{Godefroy2015}.

\subsection{Super weakly compact sets}

In the same sense as weak compactness can be understood as a localization of reflexivity (that is $X$ is reflexive if and only if $B_X$ is weakly compact), the notion of super weak compactness was introduced as a localization of superreflexivity.

\begin{definition}
    Let $A$ be a subset of a Banach space $X$. We say that $A$ is \textit{relatively super weakly compact} if $K^{\mathcal{U}}$ is weakly compact in $X^{\mathcal{U}}$ for every free ultrafilter $\mathcal{U}$ on $\N$. We say that $A$ is \textit{super weakly compact} if it is relatively super weakly compact and weakly closed.
\end{definition}

Let us recall that $X^{\mathcal{U}}$ is the quotient of $\ell_{\infty} (X)$ by the null space $N_{\mathcal{U}} = \{(x_n)_{n \in \N} \in \ell_\infty(X): \lim_\mathcal{U} x_n = 0\}$
and $A^{\mathcal{U}}$ is the subset of $X^{\mathcal{U}}$, whose elements have representants $(x_n)_{n \in \N}$ where $x_n \in A$ for each $n \in \N$. For more details see e.g. \cite{LancienRaja2022}. Every relatively norm compact set is relatively super weakly compact (this was shown in Corollary 2.15. of \cite{Cheng2010} for convex sets, but both relative compactness and relative super weak compactness are preserved by convex hulls, see \cite{Kun2021}) and every relatively super weakly compact set is clearly relatively weakly compact.

It follows that a Banach space $X$ is superreflexive if and only if $B_X$ is super weakly compact. However, there are super weakly compact sets which do not embed in superreflexive spaces \cite[Example 3.11]{Raja2016}. 

Super weakly compact sets admit many characterizations and have some very nice properties. We name a few, first for the case of convex sets. The following proposition is Proposition 2.1., Theorem 3.5. and Theorem 3.8. of \cite{LancienRaja2022}, for unexplained notions see \cite{LancienRaja2022}.

\begin{prop} \label{Proposition-SWCC-char}
    Let $A$ be a bounded closed convex subset of a Banach space $X$. The following are equivalent:
    \begin{enumerate}
        \item $A$ is super weakly compact;
        \item $A$ is finitely dentable;
        \item $A$ does not have the finite tree property;
        \item $A$ supports a bounded uniformly convex function;
        \item $X$ has an equivalent norm whose square is uniformly convex on $A$;
        \item Every real Lipschitz (or uniformly continuous) function on $A$ can be uniformly approximated by differences of convex Lipschitz functions;
        \item There is no $\theta > 0$ such that for every $n \in \N$ there are sequences $(x^n_k)_{k \leq n} \subseteq A$, $(f^n_k)_{k \leq n} \subseteq B_{X^*}$ such that $f^n_k(x^n_j) = 0$ if $k > j$ and $f^n_k(x^n_j) = \theta$ if $k \leq j$.
    \end{enumerate}
\end{prop}

We will later use the equivalence \textit{(1)} $\Leftrightarrow$ \textit{(7)}. The other points were mentioned to illustrate some strong properties of convex super weakly compact sets.

In the absence of convexity, super weakly compact sets do not behave as nicely. However, K. Tu showed in \cite{Kun2021} that the closed convex hull of a relatively super weakly compact set is super weakly compact, so some of the results can be transferred to the non-convex case as well (the paper \cite{LancienRaja2022} is dedicated to this). In particular, in Corollary 2.4. of \cite{LancienRaja2022} the following is shown.

\begin{prop} \label{Proposition-SWC-char}
    Let $A$ be a super weakly compact subset of a Banach space $X$. Then
    \begin{enumerate}
        \item $A$ is uniformly Eberlein;
        \item $A$ has the Banach-Saks property;
        \item $A$ is finitely dentable.
    \end{enumerate}
    None of the above can be reversed.
\end{prop}

For some further characterization using equi bi-Lipschitz embeddability of some families of metric spaces see Section 5 of \cite{LancienRaja2022}.

For a quantitative approach and further results, see \cite{Kun2021} and references therein. We will need one of these results, namely a measure of relative super weak compactness introduced in \cite{Kun2021} which is a quantitative analogue to the classical Grothendieck characterization of relative weak compactness.

\begin{definition}
    Let $A$ be a bounded subset of a Banach space $X$. We define
    $$\sigma(A) = \inf \{t > 0: A \subseteq K + tB_X, K \text{ is super weakly compact}\}.$$
\end{definition}

In the next proposition we list some properties of the quantity $\sigma$ we will need in further proofs, namely to show compact reduction for property (R) in Lipschitz-free spaces (Proposition \ref{Proposition-Comp-Red-R}).

\begin{prop} \label{Proposition-Sigma}
    Let $A,B$ be bounded subsets of a Banach space $X$. Then:
    \begin{enumerate}
        \item $\sigma(A) = 0$ if and only if $A$ is relatively super weakly compact;
        \item $\sigma(A+B) \leq \sigma(A) + \sigma(B)$;
        \item $\sigma(A) \leq \operatorname{diam} A$.
    \end{enumerate}
\end{prop}

\begin{proof}
    The first two properties are shown in Theorem 4.3. of \cite{Kun2021}. For the third property note that $A \subseteq \{a\} + (\operatorname{diam}A) \cdot B_X$ for any $a \in A$, and that every singleton is clearly super weakly compact.
\end{proof}

\section{Property (R)}

The following property is inspired by a well known result of Rosenthal \cite{Rosenthal1973}: \textit{Any subspace of $L_1(\mu)$ either contains $\ell_1$ or embeds in a superreflexive space}. We consider a localised version of this property for sets instead of subspaces.

\begin{definition}
    We say that a Banach space $X$ has \textit{property (R)} if every weakly precompact set in $X$ is relatively super weakly compact.
\end{definition}

Hence, in Banach spaces satisfying property (R), the only obstacle for relative super weak compactness comes from the presence of the canonical basis of $\ell_1$. In view of Propositions \ref{Proposition-WSC-char} and \ref{Proposition-Schur-char}, property (R) can be considered as an intermediate property stronger than weak sequential completeness but weaker than the Schur property.

\begin{prop} \label{Proposition-R-char}
    Let $X$ be a Banach space. Then $X$ has property (R) if and only if $X$ is weakly sequentially complete and every weakly compact subset of $X$ is super weakly compact.
\end{prop}

\begin{proof}
    This follows immediately from the definition and Proposition \ref{Proposition-WSC-char}.
\end{proof}

\begin{example}
    The following Banach spaces have property (R): superreflexive spaces; Schur spaces; $L_1(\mu)$; $C(K)^*$; preduals of JBW$^*$ triples.
\end{example}

\begin{proof}
    In superreflexive spaces every bounded set is relatively super weakly compact, so they clearly have property (R). On the other hand, in a Schur space every weakly precompact set is relatively norm compact, and thus relatively super weakly compact, so Schur spaces have property (R). The other cases follow from Proposition 6.1. of \cite{LancienRaja2022} and weak sequential completeness of the given space.
\end{proof}

Note that superreflexive spaces are, in a sense, as far from Schur spaces as possible. It could be reasoned that superreflexive spaces enjoy property (R) because it is easy for a subset to be relatively super weakly compact. On the other hand, Schur spaces enjoy property (R) because it is hard for a subset to not contain a sequence equivalent to the canonical $\ell_1$ basis.

In the next proposition we point out some properties of spaces with property (R). We say that a subset $A$ of a Banach space $X$ \textit{contains $\ell_1^N$ bases uniformly} if there is $C > 0$ such that $A$ contains $C$-equivalent copy of the canonical basis of $\ell_1^N$ for every $N \in \N$. It is a well known result by Pisier that $X$ contains $\ell_1^N$ bases uniformly if and only if $X$ has trivial type (see e.g. Theorem 13.3. of \cite{DiestelJarchowTonge1995}).

\begin{prop} \label{Proposition-Properties-R}
    Let $X$ be a Banach space with property (R). Then each of the following holds true:
    \begin{enumerate}
        \item A bounded subset of $X$ contains $\ell_1^N$ bases uniformly if and only if it contains the canonical basis of $\ell_1$.
        \item Any subspace of $X$ with nontrivial type is superreflexive.
    \end{enumerate}
\end{prop}

\begin{proof}
    The point \textit{(1)} follows as any set that contains $\ell_1^N$ bases uniformly is not relatively super weakly compact. The point \textit{(2)} follows from \textit{(1)} and the above-mentioned result of Pisier.
\end{proof}

Recall that a bounded subset $A$ of a Banach space $X$ is a \textit{(weak) Banach-Saks set} if every (resp. every weakly convergent) sequence in $A$ has Cesaro summable\footnote{Recall that a sequence $(x_n)_{n \in \N}$ is Cesaro summable if and only if the sequence $(n^{-1} \sum_{k=1}^n x_k)_{n \in \N}$ of its arithmetic means is norm convergent} subsequence. We say that a Banach space $X$ has the \textit{(weak) Banach-Saks} property if $B_X$ is a (weak) Banach-Saks set.

\begin{prop} \label{Proposition-R-WeakBS}
    Let $X$ be a Banach space in which every relatively weakly compact set is relatively super weakly compact (e.g. a Banach space with property (R)). Then $X$ has the weak Banach-Saks property.
\end{prop}

\begin{proof}
    Let $(x_n)_{n \in \N}$ be a weakly convergent sequence in $X$. Then the set $\{x_n: n \in \N\}$ is relatively weakly compact and, by the assumption, relatively super weakly compact. We now appeal to Proposition \ref{Proposition-SWC-char} \textit{(2)} and see that $\{x_n : n \in \N\}$ is a Banach-Saks set, which implies that $(x_n)_{n \in \N}$ has a Cesaro summable subsequence.
\end{proof}

Immediately we get the following corollary.

\begin{corollary} \label{Corollary-R}
    Let $X$ be a Banach space with property (R). Then none of the following spaces can embed in $X$.
    \begin{enumerate}
        \item $c_0$, non-reflexive quasi-reflexive Banach spaces;
        \item the Schreier space and the Baernstein space;
        \item $(\bigoplus_{N=1}^\infty \ell_1^N)_{\ell_p}$, $(\bigoplus_{N=1}^\infty \ell_\infty^N)_{\ell_p}$, $p > 1$.
    \end{enumerate}
\end{corollary}

\begin{proof}
    All these items follow from Propositions \ref{Proposition-Properties-R} and \ref{Proposition-R-WeakBS}.
    The first case follows as these spaces are not weakly sequentially complete. The second case follows as these spaces fail the weak Banach-Saks property. The case of the $\ell_p$ sums follows from Proposition \ref{Proposition-Properties-R} as these spaces contain $\ell_1^N$ bases uniformly but do not contain $\ell_1$ as they are reflexive. 
\end{proof}

\begin{remark}
    Note that $X = (\bigoplus_{N=1}^\infty \ell_\infty^N)_{\ell_1}$ is a Schur space (as is any $\ell_1$ sum of finite dimensional spaces), and thus has property (R). Further, $X$ contains $\ell_\infty^N$'s isometrically, and so every Banach space is finitely representable in $X$ (as every Banach space is finitely representable in $c_0$ \cite[Example 11.1.2]{nigel2006topics} and $c_0$ is clearly finitely representable in every Banach space which contains isometric copies of $\ell_\infty^N$'s). It follows that property (R) does not imply any nontrivial super-property.
\end{remark}

\section{Property (R) in Lipschitz-free spaces over superreflexive spaces}

In this section we will prove that Lipschitz-free spaces over superreflexive spaces enjoy property (R). For this we will adapt a method by Bourgain \cite{Bourgain1983,Bourgain1984}, which was modified and used in \cite{CuthDouchaWojtasczyk2016} to prove weak sequential completeness of Lipschitz-free spaces over $\R^n$, and further modified in \cite{KochanekPernecka} to extend the result to Lipschitz-free spaces over compact subsets of superreflexive spaces. The main ingredient of our proof is the observation that a certain property, which was used in an essential way in proofs of both \cite{CuthDouchaWojtasczyk2016,KochanekPernecka}, is satisfied not only for subsets that are not relatively weakly compact, but also by subsets which are not relatively super weakly compact. Hence, we will be able to strengthen the results and get relative super weak compactness instead of just relative weak compactness in the end. Later in this section, we will prove an analogue of compact reduction from \cite{AliagaCompRed} for property (R), which will allow us to get the result for superreflexive spaces instead of compact subsets of superreflexive spaces.

Before proceeding, let us recall the definition of WUC series.

\begin{definition}
    Let $(x_n)_{n \in \N}$ be a sequence in a Banach space $X$. We say that the formal series $\sum_{n=1}^\infty x_n$ is \textit{WUC (weakly unconditionally Cauchy)}, if for every $x^* \in X^*$ the series of scalars $\sum_{n=1}^\infty |x^*(x_n)|$ is convergent.
\end{definition}

The following proposition is a modification of Theorem 9 of \cite{KochanekPernecka}.
Since, except for the very first step, the proof of our proposition is identical to the proof of Theorem 9 of \cite{KochanekPernecka}, we will present only the first part and refer the reader to \cite{KochanekPernecka} for the rest of the proof.

\begin{prop} \label{Prop-KP}
    Let $M$ be a compact subset of a superreflexive space $X$. Suppose that $\Gamma \subseteq \F(M)$ is closed convex bounded and not super weakly compact. Then there is a WUC series $\sum_{k=1}^\infty \rho_k$ in $\operatorname{Lip}_0(M)$ such that
    \begin{align*}
        \limsup_{k \rightarrow \infty} \sup \{|\langle \mu, \rho_k \rangle|: \mu \in \Gamma\} > 0.
    \end{align*}
\end{prop}

\begin{proof}
    Following the proof of Theorem 9 of \cite{KochanekPernecka} we can assume, possibly by renorming $X$ and passing to a subspace, that $X$ is $p$-smooth for some $p \in (1,2]$ and that $X$ is separable. Further, by translating the set $M$, we can suppose that $0 \in M$.
    
    It was shown in the proof of Theorem 9 of \cite{KochanekPernecka} that if we can find $C \geq 1$, $\xi > 0$ and for $n \in \N$ sequences
    \begin{align*}
        &(f^n_j)_{j \leq n} \subseteq \operatorname{Lip}_0(M) \\
        &(\mu^n_j)_{j \leq n} \subseteq \operatorname{span} \{\delta(p) : p \in M\}
    \end{align*}
    such that the following conditions are satisfied:
    \begin{enumerate}
        \item $\norm{\mu^n_j}_\mathcal{F} \leq C$ for $n \in \N$ and $j \leq n$;
        \item $\norm{f^n_j}_{\operatorname{Lip}} \leq 1$ for $n \in \N$ and $j \leq n$;
        \item $|\langle \mu^n_j,f^n_k \rangle| \leq \frac{\xi}{3}$ for $n \in \N$ and $j < k \leq n$;
        \item $\langle \mu^n_j,f^n_k \rangle \geq \xi$ for $n \in \N$ and $k \leq j \leq n$;
        \item $\operatorname{dist}(\mu^n_j,\Gamma) \leq \frac{6C\xi}{\xi + 48C}$ for $n \in \N$ and $j \leq n$,
    \end{enumerate}
    then there is a WUC series $\sum_{k=1}^\infty \rho_k$ in $\operatorname{Lip}_0(M)$ with the required property. We are going to show how to find such $(f^n_j)_{j \leq n}$ and $(\mu^n_j)_{j \leq n}$ and refer the reader to Theorem 9 of \cite{KochanekPernecka} for the rest of the proof. Let us stress that the rest of the proof of Theorem 9 of \cite{KochanekPernecka} does not use that $\Gamma$ is not relatively weakly compact.
    
    So, suppose that $\Gamma$ is as above. We can use Theorem 3.8. of \cite{LancienRaja2022} (that is the equivalence \textit{(1)} $\Leftrightarrow$ \textit{(7)} of our Proposition \ref{Proposition-SWCC-char}) to find $\theta > 0$ and for every $n \in \N$ finite sequences $(x^n_k)_{k \leq n} \subseteq \Gamma$ and $(f^n_k)_{k \leq n} \subseteq B_{\operatorname{Lip}_0(M)}$ such that $f^n_j (x^n_k) = 0$ for $j > k$ and $f^n_j (x^n_k) = \theta$ for $j \leq k$.
    
    Let $C = 1 + \sup \{\norm{\gamma}_\mathcal{F}: \gamma \in \Gamma\}$ and $0 < \xi < \theta$. Fix $\epsilon > 0$ small enough such that
    $$\epsilon < \min \left\{\frac{6C\xi}{\xi + 48C}, \frac{\xi}{3}, \theta - \xi, 1 \right\}.$$
    We can use density of $\operatorname{span} \{\delta(p) : p \in M\}$ in $\mathcal{F}(M)$ to find, for each $n \in \N$ and $k \leq n$, an element $\mu^n_k \in \operatorname{span} \{\delta(p) : p \in M\}$ such that $\norm{\mu^n_k - x^n_k}_{\mathcal{F}} \leq \epsilon$. The required properties (1)-(5) then clearly follow from the properties of $x^n_k$'s and the choice of $\epsilon$.
\end{proof}

\begin{theorem} \label{Theorem-Comp-R}
    Let $M$ be a compact subset of a superreflexive space $X$. Then the Lipschitz-free space $\mathcal{F}(M)$ has property (R).
\end{theorem}

\begin{proof}
    Let $\Gamma' \subseteq \mathcal{F}(M)$ be a bounded subset which is not relatively super weakly compact. Set $\Gamma = \overline{\operatorname{conv}} \Gamma'$. As $\Gamma$ is not super weakly compact, by Proposition \ref{Prop-KP} there is a WUC series $\sum_{k=1}^\infty \rho_k$ in $\operatorname{Lip}_0(M)$ such that 
    $$\lambda = \limsup_{k \rightarrow \infty} \sup \{|\langle \mu, \rho_k \rangle|: \mu \in \Gamma\} > 0.$$

    As $\sum_{k=1}^\infty \rho_k$ is WUC, the operator $T: \mathcal{F}(M) \rightarrow \ell_1$, $T(\mu) = (\langle \mu,\rho_k \rangle)_{k \in \N}$, is well defined and bounded. It follows from the fact that $\lambda > 0$ that $T(\Gamma) \subseteq \ell_1$ is a bounded set which is not relatively norm compact. By Proposition \ref{Proposition-Schur-char} and Schur property of $\ell_1$, $T(\Gamma)$ is not weakly precompact. As weak precompactness is preserved under bounded linear images, $\Gamma$ is not weakly precompact either.

    Hence, $\Gamma$ is not weakly precompact, and so neither is $\Gamma'$ (as being weakly precompact is preserved under convex hulls and closures, see Addendum of \cite{Rosenthal1977}).
\end{proof}

\begin{remark}
    In Theorem III.D.31 of \cite{WojtaszczykBook} it is shown that the dual of a rich subspace $X$ of $C(K,\R^n)$ is weakly sequentially complete. In the same way as above, the proof can be modified to show that $X^*$ has property (R). Indeed, the condition $(*)$ on the bottom of page 169 of \cite{WojtaszczykBook} follows not only if the set not relatively weakly compact, but also when the given set is not relatively super weakly compact.
\end{remark}

Now we prove the promised compact reduction for property (R).

\begin{prop} \label{Proposition-Comp-Red-R}
    Let $M$ be a complete metric space. Then $\mathcal{F}(M)$ has property (R) if and only if $\mathcal{F}(K)$ has property (R) for every compact $K \subseteq M$.
\end{prop}

\begin{proof}
    As property (R) is clearly passed to subspaces and $\mathcal{F}(K)$ is a subspace of $\mathcal{F}(M)$ for $K \subseteq M$, we need to show only the backwards implication. Let $W \subseteq \mathcal{F}(M)$ be weakly precompact. We will show that $\sigma(W) = 0$ and apply Proposition \ref{Proposition-Sigma} to show that $W$ is relatively super weakly compact. Fix $\epsilon > 0$ and use compact reduction for weakly precompact sets (Theorem 2.3. of \cite{AliagaCompRed}) to find compact $K \subseteq M$ and linear (not necessarily bounded) $T: \operatorname{span} W \rightarrow \mathcal{F}(K)$ such that
    \begin{enumerate}
        \item $\norm{\mu - T \mu}_{\mathcal{F}} \leq \epsilon$ for all $\mu \in W$,
        \item There are bounded linear maps $T_k: \mathcal{F}(M) \rightarrow \mathcal{F}(M)$, $k \in \N$, such that $(T_k|W)_{k \in \N}$ converges to $T|W$ uniformly on $W$.
    \end{enumerate}
    
    First we show that $T(W)$ is weakly precompact in $\mathcal{F}(K)$. As weak precompactness is preserved by bounded linear maps, for every $k \in \N$ the set $T_k(W)$ is weakly precompact. Suppose, for a contradiction, that $T(W)$ contains a sequence $C$-equivalent to the canonical basis of $\ell_1$ for some $C > 0$. Using (2) we find $k \in \N$ such that $T_k(W)$ contains a sequence $C/2$-equivalent to the canonical basis of $\ell_1$, contradicting that $T_k(W)$ is weakly precompact. Hence, $T(W)$ is weakly precompact as required.

    But $T(W) \subseteq \mathcal{F}(K)$, so by the assumption $T(W)$ is relatively super weakly compact. It follows from Proposition \ref{Proposition-Sigma} \textit{(1)} that $\sigma(T(W)) = 0$. Further, by (1) we get that $W \subseteq T(W) + \epsilon B_{\mathcal{F}(M)}$, and thus by Proposition \ref{Proposition-Sigma} \textit{(2)} and \textit{(3)} $$\sigma(W) \leq \sigma(T(W)) + \sigma(\epsilon B_{\mathcal{F}(M)}) \leq 0 + \operatorname{diam}(\epsilon B_{\mathcal{F}(M)}) = 2 \epsilon.$$
    As $\epsilon > 0$ was arbirary, $\sigma(W) = 0$ and $W$ is relatively super weakly compact by Proposition \ref{Proposition-Sigma} \textit{(1)} as required to show that $\mathcal{F}(M)$ has property (R).
\end{proof}

Combining the results above, we get the main theorem of this section.

\begin{theorem} \label{Theorem:Main}
    Let $X$ be a superreflexive space. Then $\mathcal{F}(X)$ has property (R).
\end{theorem}

\begin{proof}
    By Proposition \ref{Proposition-Comp-Red-R} it is enough to show that $\mathcal{F}(K)$ has property (R) for every $K \subseteq X$ compact. But this follows from Theorem \ref{Theorem-Comp-R}.
\end{proof}

\begin{remark}
    It follows from Theorem \ref{Theorem:Main} and Corollary \ref{Corollary-R} that $\mathcal{F}((\bigoplus_{N=1}^\infty \ell_1^N)_{\ell_2})$ is not isomorphic to a subspace of $\mathcal{F}(\ell_2)$, and also not isomorphic to a subspace of $\mathcal{F}(Y)$ for finite-dimensional $Y$. This was apparently unknown before the results of this paper.
\end{remark}

\section{Additional remarks and open problems}

\subsection{Compact reduction of weak Banach-Saks property}

In what follows we show that compact reduction for the weak Banach-Saks property is also true. This observation and its proof is due to R. Aliaga, G. Grelier and A. Procházka. We thank them for kindly letting us to include it in this paper.

\begin{prop}
    Let $M$ be a complete metric space. Then $\mathcal{F}(M)$ has the weak Banach-Saks property if and only if $\mathcal{F}(K)$ has the weak Banach-Saks property for every compact $K \subseteq M$.
\end{prop}

\begin{proof}
    Let $(\mu_n)_{n \in \N} \subseteq B_{\mathcal{F}(M)}$ be weakly null and let $(\epsilon_n)_{n \in \N} \subseteq (0,1)$ be a sequence decreasing to 0.
    Then the set $W = \{\mu_n : n \in \N\}$ is weakly precompact, and thus by Theorem~2.3 of~\cite{AliagaCompRed} there is a compact set $K_1 \subseteq M$ and linear (not necessarily bounded) $T: \operatorname{span} W \rightarrow \mathcal{F}(K_1)$ such that
    \begin{enumerate}
        \item $\norm{\mu_n - T \mu_n}_{\mathcal{F}} \leq \epsilon$ for all $n \in \N$,
        \item There are bounded linear maps $T_k: \mathcal{F}(M) \rightarrow \mathcal{F}(M)$, $k \in \N$, such that $(T_k|W)_{k \in \N}$ converges to $T|W$ uniformly on $W$.
    \end{enumerate}
    Set $\xi_n^1 = T \mu_n$ for $n \in \N$, then the sequence $(\xi_n^1)_{n \in \N}$ is weakly null (see the proof of Corollary 2.4 of \cite{AliagaCompRed}). Since $\mathcal{F}(K)$ has the weak Banach-Saks property, we find an infinite subset $N_1 \subseteq \N$ such that all further subsequences of $(\xi^1_n)_{n \in N_1}$ are Cesaro summable (the fact that we can choose a subsequence such that \textit{all} its further subsequences are Cesaro summable follows from \cite{ErdosMagidor1976}).
    We now apply the same procedure for $(\mu_n)_{n\in N_1}$ and $\epsilon_2$ and obtain infinite $N_2 \subseteq N_1$ and $(\xi_n^2)_{n \in N_1} \subseteq \mathcal{F}(K_2)$ such that $\norm{\mu_n-\xi_n^2}_{\mathcal{F}} \leq \epsilon_2$ for all $n \in N_1$, and such that every subsequence of $(\xi_n^2)_{n\in N_2}$ is Cesaro summable. We can also suppose that $\min N_1 \notin N_2$.
    Following the same scheme we get
    \begin{itemize}
        \item a decreasing sequence $(N_j)_{j \in \N}$ of infinite subsets of $\N$ with $\min N_j \notin N_{j+1}$ for $j \in \N$,
        \item for $j \in \N$ sequences $(\xi_n^j)_{n \in N_j}$ such that
        \begin{enumerate}[i)]
            \item Every subsequence of $(\xi_n^j)_{n\in N_j}$ is Cesaro summable.
            \item $\norm{\mu_n-\xi_n^j}_{\mathcal{F}} \leq \epsilon_j$ for all $n \in N_j$.
        \end{enumerate}
    \end{itemize}
    
    Take $M = \{\min N_i:i \in \N \}$.  
    We claim that $(\mu_n)_{n \in M}$ is Cesaro summable.
    Indeed, fix $\epsilon>0$ and let $n \in \N$ be such that $\epsilon_n \leq \epsilon/3$.
    Let us enumerate $M=\{m_1<m_2<\ldots\}$.
    Notice that $(m_k)_{k=n}^\infty \subseteq N_n$ for every $n \in \N$. Hence, for $k \geq n$ we have

    \begin{align*}
        \norm{\frac{\mu_{m_1}+\cdots+ \mu_{m_{n-1}}+\mu_{m_n}+\cdots+ \mu_{m_k}}{k}} &\leq \frac{n-1}{k} + \norm{\frac{\mu_{m_n} + \cdots + \mu_{m_{k}}}{k-n+1}} \\
        & \leq \frac{n-1}{k} + \norm{\frac{\xi^n_{m_n} + \cdots + \xi^n_{m_{k}}}{k-n + 1}} + \epsilon_n,
    \end{align*}
    where the second inequality follows from ii). Now we can use i) to see that for $k$ large enough so that $\frac{n-1}{k} < \epsilon/3$ and $\norm{\frac{\xi^n_{m_n} + \cdots + \xi^n_{m_{k}}}{k-n + 1}} <\epsilon/3$ we have that $\norm{\frac{\mu_{m_1}+ \cdots + \mu_{m_k}}{k}} < \epsilon$, proving that $(\mu_n)_{n \in M}$ is Cesaro summable.
\end{proof}

\subsection{Open problems}

Note that all the examples of Banach spaces in which every relatively weakly compact set is relatively super weakly compact mentioned above are weakly sequentially complete. This inspires the following question:

\begin{question} \label{Question-R-noWSC}
    Is there a Banach space in which every relatively weakly compact set is relatively super weakly compact but which is not weakly sequentially complete?
\end{question}
\noindent A natural candidate to test for Question \ref{Question-R-noWSC} is the James's space or some of its variants.

The next problem is concerned about super-variants of some properties studied in this paper.

\begin{question} \label{Question-super}
    Is the space $\mathcal{F}(X)$ super weakly sequentially complete (does it have super weak Banach-Saks property), for $X$ superreflexive? What about $\mathcal{F}(\R^2)$?  
\end{question}
Positive answer to Question \ref{Question-super} would give us that $\mathcal{F}(\ell_2)$ has nontrivial cotype, solving the dichotomy of \cite{GarciaGrelier2023} (see the beginning of Section 5 of \cite{GarciaGrelier2023}). 

\bibliographystyle{siam}
\bibliography{bibliography}

\begin{thebibliography}{10}

\bibitem{nigel2006topics}
{\sc F.~Albiac and N.~Kalton}, {\em Topics in Banach Space Theory}, no.~sv. 10 in Graduate Texts in Mathematics, Springer, 2006.

\bibitem{AliagaCompRed}
{\sc R.~J. Aliaga, C.~No{\^u}s, C.~Petitjean, and A.~Proch{\'a}zka}, {\em Compact reduction in {Lipschitz}-free spaces}, Stud. Math., 260 (2021), pp.~341--359.

\bibitem{Bourgain1983}
{\sc J.~Bourgain}, {\em On weak completeness of the dual of spaces of analytic and smooth functions}, Bull. Soc. Math. Belg., S{\'e}r. B, 35 (1983), pp.~111--118.

\bibitem{Bourgain1984}
\leavevmode\vrule height 2pt depth -1.6pt width 23pt, {\em The {Dunford}-{Pettis} property for the ball-algebras, the polydisc-algebras and the {Sobolev} spaces}, Stud. Math., 77 (1984), pp.~245--253.

\bibitem{Cheng2018}
{\sc L.~Cheng, Q.~Cheng, S.~Luo, K.~Tu, and J.~Zhang}, {\em On super weak compactness of subsets and its equivalences in {Banach} spaces}, J. Convex Anal., 25 (2018), pp.~899--926.

\bibitem{Cheng2010}
{\sc L.~Cheng, Q.~Cheng, B.~Wang, and W.~Zhang}, {\em On super-weakly compact sets and uniformly convexifiable sets}, Stud. Math., 199 (2010), pp.~145--169.

\bibitem{CuthDouchaWojtasczyk2016}
{\sc M.~C{\'u}th, M.~Doucha, and P.~Wojtaszczyk}, {\em On the structure of {Lipschitz}-free spaces}, Proc. Am. Math. Soc., 144 (2016), pp.~3833--3846.

\bibitem{DiestelJarchowTonge1995}
{\sc J.~Diestel, H.~Jarchow, and A.~Tonge}, {\em Absolutely summing operators}, vol.~43 of Camb. Stud. Adv. Math., Cambridge: Cambridge Univ. Press, 1995.

\bibitem{ErdosMagidor1976}
{\sc P.~Erd{\H{o}}s and M.~Magidor}, {\em A note on regular methods of summability and the {Banach}-{Saks} property}, Proc. Am. Math. Soc., 59 (1976), pp.~232--234.

\bibitem{GarciaGrelier2023}
{\sc L.~C. Garc{\'{\i}}a-Lirola and G.~Grelier}, {\em Lipschitz-free spaces, ultraproducts, and finite representability of metric spaces}, J. Math. Anal. Appl., 526 (2023), p.~14.
\newblock Id/No 127253.

\bibitem{Godefroy2015}
{\sc G.~Godefroy}, {\em A survey on {Lipschitz}-free {Banach} spaces}, Commentat. Math., 55 (2015), pp.~89--118.

\bibitem{GrelierRaja2022}
{\sc G.~Grelier and M.~Raja}, {\em On uniformly convex functions}, J. Math. Anal. Appl., 505 (2022), p.~25.
\newblock Id/No 125442.

\bibitem{KochanekPernecka}
{\sc T.~Kochanek and E.~Perneck{\'a}}, {\em Lipschitz-free spaces over compact subsets of superreflexive spaces are weakly sequentially complete}, Bull. Lond. Math. Soc., 50 (2018), pp.~680--696.

\bibitem{LancienRaja2022}
{\sc G.~Lancien and M.~Raja}, {\em Nonlinear aspects of super weakly compact sets}, Ann. Inst. Fourier, 72 (2022), pp.~1305--1328.

\bibitem{Raja2016}
{\sc M.~Raja}, {\em Super {WCG} {Banach} spaces}, J. Math. Anal. Appl., 439 (2016), pp.~183--196.

\bibitem{Rosenthal1973}
{\sc H.~P. Rosenthal}, {\em On subspaces of {L}{{\(^p\)}}}, Ann. Math. (2), 97 (1973), pp.~344--373.

\bibitem{Rosenthal1974}
\leavevmode\vrule height 2pt depth -1.6pt width 23pt, {\em A characterization of {B}anach spaces containing $l_1$}, Proceedings of the National Academy of Sciences, 71 (1974), pp.~2411--2413.

\bibitem{Rosenthal1977}
\leavevmode\vrule height 2pt depth -1.6pt width 23pt, {\em Point-wise compact subsets of the first {Baire} class}, Am. J. Math., 99 (1977), pp.~362--378.

\bibitem{Kun2021}
{\sc K.~Tu}, {\em Convexification of super weakly compact sets and measure of super weak noncompactness}, Proc. Am. Math. Soc., 149 (2021), pp.~2531--2538.

\bibitem{Weaver1999}
{\sc N.~Weaver}, {\em Lipschitz algebras}, Singapore: World Scientific, 1999.

\bibitem{WojtaszczykBook}
{\sc P.~Wojtaszczyk}, {\em Banach spaces for analysts}, Cambridge: Cambridge Univ. Press, 1996.

\end{thebibliography}

\end{document}